\documentclass[letterpaper,10pt,conference]{ieeeconf}
\IEEEoverridecommandlockouts
\usepackage{amsmath,amssymb,mathtools,bm}
\usepackage{graphicx}
\usepackage{booktabs,multirow,array}
\usepackage{tikz}
\usepackage{balance}
\usepackage{cuted}
\let\labelindent\relax 
\usepackage{enumitem}

\makeatletter
\let\NAT@parse\undefined
\makeatother

\makeatletter
\let\NAT@parse\undefined
\makeatother

\usepackage[numbers,sort&compress]{natbib}

\usetikzlibrary{arrows.meta,positioning,fit}

\input{requirements}
\def\[#1\]{\begin{align*}#1\end{align*}}

\newcounter{qst}
\crefname{qst}{Question}{Questions}

\providecommand{\R}{\mathbb{R}}
\providecommand{\Sbb}{\mathbb{S}}
\providecommand{\cC}{\mathcal{C}}
\providecommand{\cF}{\mathcal{F}}
\providecommand{\cS}{\mathcal{S}}
\providecommand{\cE}{\mathcal{E}}
\providecommand{\cM}{\mathcal{M}}
\providecommand{\cZ}{\mathcal{Z}}

\providecommand{\distop}{\operatorname{dist}}
\providecommand{\diamop}{\operatorname{diam}}

\providecommand{\opnorm}[1]{\left\|#1\right\|_{\mathrm{op}}}

\providecommand{\inner}[2]{\left\langle #1,#2\right\rangle}
\providecommand{\transpose}{\mathsf{T}}

\newcommand{\defeq}{\coloneqq}

\let\hat\widehat
\let\tilde\widetilde

\newcommand{\norm}[1]{\|#1\|}

\makeatletter

  \usepackage{amsthm}

\newtheorem{theorem}{Theorem}
\newtheorem{lemma}[theorem]{Lemma}
\newtheorem{proposition}[theorem]{Proposition}
\newtheorem{corollary}[theorem]{Corollary}

\newtheorem{definition}[theorem]{Definition}
\newtheorem{assumption}[theorem]{Assumption}

\newtheorem{remark}[theorem]{Remark}

\usepackage{stmaryrd}

\usepackage{xcolor}

\usepackage[linesnumbered,ruled,lined,noend]{algorithm2e}

\NewDocumentCommand{\numberthis}{om}{%
  \IfNoValueTF{#1}{%
    \refstepcounter{equation}\tag{\theequation}%
  }{%
    \tag{#1}%
  }%
  \label{#2}%
}

\definecolor{darkgrey}{gray}{0.3}
\definecolor{commentcolor}{gray}{0.5}
\SetKwComment{Hline}{}{\vspace{-3mm}\textcolor{gray}{\hrule}\vspace{1mm}}
\SetKwComment{Comment}{\color{commentcolor}[$\triangleright$\ }{}
\SetCommentSty{}
\SetNlSty{}{\color{darkgrey}}{}
\SetKwProg{Fn}{function}{}{}
\SetKwProg{Subr}{subroutine}{}{}
\crefalias{AlgoLine}{line}%
\crefname{algocf}{Algorithm}{Algorithms}

\title{\LARGE\bf
Algorithmic Optimality Guarantees for Nonsmooth $H_\infty$ Output-Feedback Policy Search}

\author{Ashkan Soleymani and Patrick Jaillet%
\thanks{Department of Electrical Engineering and Computer Science and the Laboratory for Information and Decision Systems, Massachusetts Institute of Technology, Cambridge, MA 02139, USA. E-mail addresses: {\tt\small ashkanso@mit.edu} and {\tt\small jaillet@mit.edu}.}}

\begin{document}
\maketitle
\thispagestyle{empty}
\pagestyle{empty}

\begin{abstract}
We study continuous-time full-order dynamic output-feedback $H_\infty$ policy search, a nonconvex and nonsmooth problem. Direct policy search is a central paradigm in reinforcement learning and continuous control, but rigorous guarantees remain scarce in robust output-feedback settings. The $H_\infty$ problem is a canonical benchmark because it captures disturbance attenuation and robustness while exposing the hard nonsmooth geometry of policy-space optimization. We prove that on the exact identity-gauge slice of the extended convex lift, $\varepsilon$-stationarity yields $O(\varepsilon)$-suboptimality on compact exact slices, which in turn yields convergence-rate guarantees for nonsmooth policy-search methods. This result addresses the finite-time optimality-gap question raised by \citet{guo2022global} in the more general dynamic output-feedback $H_\infty$ policy-search setting. We further use the established value equivalence supplied by extended convex lifting to formulate a nonstrict-feasibility bisection method with one final strict-feasibility recovery step, yielding an explicit $\varepsilon$-optimal stabilizing controller. These results provide a quantitative and algorithmic strengthening of prior qualitative optimality theory for nonsmooth $H_\infty$ policy search.
\end{abstract}

\vspace{-0.15cm}
\section{Introduction}
Direct policy search, or policy optimization, refers to optimizing a parameterized controller directly against the closed-loop performance objective. It is a core paradigm in reinforcement learning (RL) and continuous control. The policy-gradient theorem of \citet{sutton1999policy} made this viewpoint algorithmic, and later deterministic, trust-region, and proximal variants became standard tools for continuous-action RL \citep{silver2014deterministic,schulman2015trust,schulman2017proximal}. These methods have driven influential successes in simulated robotics and continuous control, including locomotion and end-to-end learning from high-dimensional observations \citep{lillicrap2015continuous,recht2019tour}. From a control viewpoint, direct policy search is attractive because it is end-to-end, accommodates richly parameterized controllers, and is compatible with model-based, model-free, and data-driven implementations \citep{fazel2018global,zhang2021policy}.

This viewpoint has recently reshaped parts of systems and control. A growing literature studies the geometry and algorithmics of policy optimization for the linear-quadratic regulator (LQR), linear-quadratic-Gaussian (LQG) control, output estimation, and robust control, revealing benign landscapes in some settings and geometric obstructions in others \citep{fazel2018global,zhang2021policy,umenberger2022globally,zheng2025extended}. The partially observed case is especially delicate. Dynamic output-feedback controllers carry their own internal state, and controllers related by similarity transformations represent the same input-output law. This redundancy creates degenerate directions, complicates stationarity, and can obstruct naive Euclidean descent \citep{hu2022connectivity,kraisler2024output}. Recent work on output estimation and LQG makes this difficulty explicit and shows that suitable reparameterizations or quotient-geometric viewpoints can restore meaningful first-order analysis \citep{umenberger2022globally,kraisler2024output}.

Among robust synthesis problems, $H_\infty$ control is a natural stress test for direct policy search. It is classical and practically central, but in policy space the objective is nonconvex and nonsmooth. In the state-feedback setting, \citet{guo2022global} proved that every Clarke-stationary stabilizing gain is globally optimal and showed that Goldstein-type methods can recover the global solution. For dynamic output feedback, \citet{tang2023global} established global optimality of Clarke-stationary controllers under a nondegenerate exact-certificate condition. More broadly, recent extended convex lifting (ECL) results reveal hidden convexity behind optimal and robust control problems and show that nondegenerate first-order stationary points are globally optimal across several benchmark settings, including output-feedback $H_\infty$ synthesis \citep{zheng2026benign,zheng2025extended}. Recent work  established weak convexity and deterministic subgradient rates for discrete-time robust control, although the weak Polyak--\L{}ojasiewicz conclusion remains confined to the state-feedback case \citep{watanabe2025policy}.

The remaining gap is quantitative and algorithmic. The theorem of \citet{tang2023global} is qualitative: it certifies exact Clarke-stationary points, but it does not convert an approximate stationarity residual into a value-gap bound and, consequently, does not provide algorithms with convergence rates for finding an optimal policy. Nor does it separate the two roles played by the lifted variables. One role is global: the full nonstrict lift captures the closure of strict stabilizing upper-bound certificates and therefore should be used to compute the optimal value. The other is local: after exactification and gauge fixing, an exact identity slice provides the appropriate chart for first-order certification and, hence, a foundation for the analysis of nonsmooth optimization algorithms for policy search. Keeping these roles distinct leads to sharper geometric insight and a more usable algorithmic interpretation.

Our main message is that the ECL has complementary local and global roles; see Figure~\ref{fig:two-roles}. The stable exact identity-gauge slice is the appropriate chart for local first-order certification and, hence, for analyzing the convergence rates of nonsmooth optimization algorithms for policy search, while the full nonstrict convex lifted set computes the global infimal value through closure and strict lifted points support stabilizing controller recovery. Building on this separation, our contributions are fourfold.
\begin{enumerate}
    \item We prove an exact \emph{reconditioning} lemma showing every exact nondegenerate certificate is realization-equivalent to one with $P_{12}=I$.
    \item On that identity-gauge exact slice, we establish a pairwise first-order error bound, implying that on every compact exact slice, $\varepsilon$-stationarity yields $O(\varepsilon)$-suboptimality.
    \item Building on this suboptimality bound, we study the convergence rates of off-the-shelf nonsmooth optimization algorithms for finding Goldstein stationary points to the optimal policy in the dynamic output-feedback $H_\infty$ policy search problem.
    \item Using the established value-equivalence theorem for extended convex lifting and its dynamic output-feedback $H_\infty$ specialization \citep[Theorem~2.1 and Corollary~4.1]{zheng2025extended}, we formulate a nonstrict-feasibility bisection procedure in the present generalized-plant coordinates. One final strict-feasibility recovery step returns an explicit $\varepsilon$-optimal stabilizing controller.
\end{enumerate}

\emph{Relation to classical $H_\infty$ synthesis:} Classical Riccati and linear matrix inequality (LMI) methods already solve full-order $H_\infty$ synthesis in model-based settings~\citep{doyle1988state,masubuchi1998lmi}. Our focus is instead the nonsmooth policy-optimization viewpoint, where the extended convex lift enables quantitative first-order guarantees and algorithmic controller recovery. This perspective is particularly relevant to optimization-based, certificate-driven and potentially data-driven methods.

\begin{figure*}[t]
\centering
\begin{tikzpicture}[
    scale=0.88,
    transform shape,
    >=Latex,
    node distance=5mm and 6mm,
    every node/.style={font=\footnotesize},
    box/.style={
        draw,
        rounded corners,
        align=center,
        minimum width=3.2cm,
        minimum height=0.9cm
    }
]
  \node[box] (C) {Controller Space\\$K\in\cC$};
  \node[box, right=18mm of C] (F) {Full Lift\\$z\in\cF$};
  \node[box, below=8mm of F] (E) {Exact Identity Slice\\$z\in\cE_I\subset\cF$};
  \node[box, right=14mm of F,draw=myMaroon!90!black,
  fill=myMaroon!20, line width=1pt] (G) {Global Value\\Bisection in $\gamma$};
  \node[box, right=14mm of E] (L) {Local certificate and Suboptimality\\Bound from $\distop(0,\partial J)$};
  \node[box, right=14mm of L, draw=myMaroon!90!black,
  fill=myMaroon!20, line width=1pt] (N) {Analysis of Nonsmooth \\ Optimization Algorithms};

  \draw[->, thick] (C) -- node[above] {ECL chart} (F);
  \draw[->, thick] (F) -- node[right] {Exactify/Gauge-fix} (E);
  \draw[->, thick] (F) -- (G);
  \draw[->, thick] (E) -- (L);
  \draw[->, thick] (L) -- (N);
\end{tikzpicture}
\caption{Two roles of the lift. The full set $\cF$ computes the global value, while the exact slice $\cE_I$ yields local first-order certificates and supports the analysis of convergence rates for nonsmooth optimization algorithms.}
\label{fig:two-roles}
\vspace{-1ex}
\end{figure*}
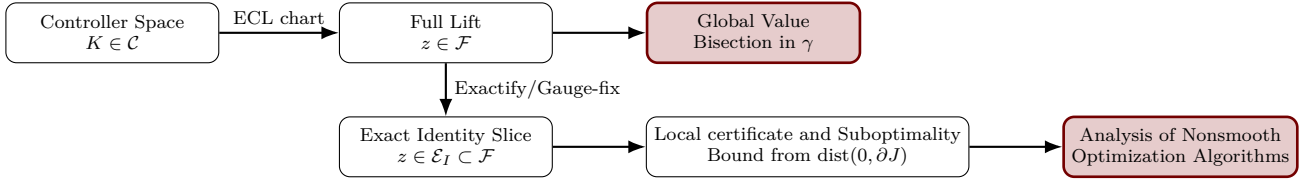

\section{Problem Definition} \label{sec:preliminaries}

\subsection{Dynamic output-feedback $H_\infty$ policy search}
We consider the continuous-time linear time-invariant plant
\begin{align}
\dot x(t) &= A x(t) + B_1 w(t) + B_2 u(t), \label{eq:plant1}\\
 z(t) &= C_1 x(t) + D_{11} w(t) + D_{12} u(t), \label{eq:plant2}\\
 y(t) &= C_2 x(t) + D_{21} w(t), \label{eq:plant3}
\end{align}
where $x(t)\in\R^{n_x}$ is the plant state, $w(t)$ is the exogenous input, $u(t)$ is the control input, $z(t)$ is the performance output, and $y(t)$ is the measured output. As in \citet{tang2023global}, we assume that $(A,B_1)$ and $(A,B_2)$ are controllable and that $(C_1,A)$ and $(C_2,A)$ are observable. We restrict attention to full-order dynamic controllers of the form
\begin{align}
\dot \xi(t) &= A_K \xi(t) + B_K y(t), \label{eq:controller1}\\
 u(t) &= C_K \xi(t) + D_K y(t), \label{eq:controller2}
\end{align}
where $\xi(t)\in\R^{n_x}$ is the controller state. We identify the controller with the block matrix
\[
K = \begin{bmatrix} D_K & C_K \\ B_K & A_K \end{bmatrix}.
\]
Throughout, controller space is viewed as a finite-dimensional Euclidean space equipped with the Frobenius inner product and its induced norm.

The closed-loop transfer matrix from $w$ to $z$ is
\[
T_{zw}(s;K)=C_{\mathrm{cl}}(K)(sI-A_{\mathrm{cl}}(K))^{-1}B_{\mathrm{cl}}(K)+D_{\mathrm{cl}}(K),
\]
where the closed-loop matrices are the standard ones induced by \eqref{eq:plant1}--\eqref{eq:controller2}; explicitly,
\begin{align*}
A_{\mathrm{cl}}(K)&=
\begin{bmatrix}
A+B_2D_KC_2 & B_2C_K \\
B_KC_2 & A_K
\end{bmatrix},\\
B_{\mathrm{cl}}(K)&=
\begin{bmatrix}
B_1+B_2D_KD_{21}\\
B_KD_{21}
\end{bmatrix},\\
C_{\mathrm{cl}}(K)&=
\begin{bmatrix}
C_1+D_{12}D_KC_2 & D_{12}C_K
\end{bmatrix},\\
D_{\mathrm{cl}}(K)&=D_{11}+D_{12}D_KD_{21}.
\end{align*}
The stabilizing set is
\[
\cC:=\{K: A_{\mathrm{cl}}(K)\text{ is Hurwitz}\},
\]
and the $H_\infty$ objective is
\[
J(K):=\norm{T_{zw}(\cdot;K)}=\|T_{zw}(\cdot;K)\|_{\infty},\qquad K\in \cC.
\]
Hence the policy search problem is
\begin{equation}\label{eq:policyopt}
J^* \defeq \inf_{K\in\cC} J(K).
\end{equation}

A key fact from \citet[Proposition~1]{tang2023global} is that $J$ is locally Lipschitz and subdifferentially regular on $\cC$.

\begin{proposition}\label{prop:TZ-regular}
For problem \eqref{eq:policyopt}, the mapping $J:\cC\to\R$ is locally Lipschitz and subdifferentially regular on $\cC$. In particular, for every $K\in\cC$ and every direction $V$, the directional derivative $J'(K;V)$ exists and satisfies
\[
J'(K;V)=J^{\circ}(K;V)=\max_{G\in\partial J(K)}\inner{G}{V},
\]
where $\partial J(K)$ denotes the Clarke subdifferential.
\end{proposition}

For locally Lipschitz regular functions, the ordinary directional derivative coincides with the Clarke directional derivative, which in turn equals the support function of the Clarke subdifferential; see, for example, \citep[Section~2.3]{clarke1990optimization} or \citep[Chapter~8]{rockafellar1998variational}. We conclude this section with the definition of the Clarke and Goldstein subdifferential and stationary points for completeness~\citep{clarke1990optimization,goldstein1977optimization}.

\begin{definition}
    Given a locally Lipschitz function $f:C \rightarrow \mathbb{R}$, its Clarke subdifferential is defined as
    \[
    \partial f(x)
    =
    \operatorname{conv}
    \left\{
    \lim_{x_k \to x} \nabla f(x_k)
    \;\middle|\;
    \nabla f(x_k)\ \text{exists},\ x_k \in \mathcal C
    \right\},
    \]
    and its Goldstein subdifferential is defined as
    \[
    \partial_\delta f(x)
    =
    \operatorname{conv}  \left\{
    \bigcup_{y \in B_2(x, \delta)} \partial f(y)
    \right\}.
    \]
    A point $x$ is called an $\varepsilon$-Clarke stationary point if $\distop(0, \partial f(x)) \leq \varepsilon$, and an $(\delta, \varepsilon)$-Goldstein stationary point if $\distop(0, \partial_\delta f(x)) \leq \varepsilon$, where $\distop$ denotes the Euclidean distance.
\end{definition}

The Clarke stationary point serves as a primary solution concept for nonsmooth optimization; however, it suffers from significant computational intractability (e.g., there is no algorithm that can find Clarke stationary points in finite time \citep[Theorem~5]{zhang2020complexity}). However, recent advances show that Goldstein stationarity, as a more relaxed solution concept, offers algorithmic tractability~\citep{zhang2020complexity,davis2022gradient}.

\begin{figure*}[!t]
\noindent\makebox[\textwidth][c]{\rule{\textwidth}{0.4pt}}
\begin{equation}\label{eq:M-def}
\resizebox{0.92\textwidth}{!}{$\displaystyle
\cM(z):=
\begin{bmatrix}
AX+B_2F+(AX+B_2F)^\transpose & M^\transpose + A + B_2LC_2 & B_1+B_2LD_{21} & (C_1X+D_{12}F)^\transpose \\
M+(A+B_2LC_2)^\transpose & YA+HC_2+(YA+HC_2)^\transpose & YB_1+HD_{21} & (C_1+D_{12}LC_2)^\transpose \\
(B_1+B_2LD_{21})^\transpose & (YB_1+HD_{21})^\transpose & -\gamma I & (D_{11}+D_{12}LD_{21})^\transpose \\
C_1X+D_{12}F & C_1+D_{12}LC_2 & D_{11}+D_{12}LD_{21} & -\gamma I
\end{bmatrix}$}
\end{equation}
\vspace{0.3em}
\noindent\makebox[\textwidth][c]{\rule{\textwidth}{0.4pt}}
\end{figure*}

\subsection{Nondegenerate bounded-real certificates and Extended Convex Lifting}
Define the matrix $ N(K,P,\gamma) := $ 
\begin{equation}\label{eq:N-def}
\begin{aligned}
\left[
\begin{array}{ccc}
A_{\mathrm{cl}}(K)^\top P + P A_{\mathrm{cl}}(K) & P B_{\mathrm{cl}}(K) & C_{\mathrm{cl}}(K)^\top \\
B_{\mathrm{cl}}(K)^\top P & -\gamma I & D_{\mathrm{cl}}(K)^\top \\
C_{\mathrm{cl}}(K) & D_{\mathrm{cl}}(K) & -\gamma I
\end{array}
\right].
\end{aligned}
\end{equation}
Following \citet{tang2023global}, let
\begin{equation}\label{eq:Snd-def}
\begin{aligned}
\cS_{\mathrm{nd}}
:=
\Bigl\{(K,P,\gamma):\;&
P=\begin{bmatrix}P_{11}&P_{12}\\P_{12}^\transpose&P_{22}\end{bmatrix}\in\Sbb^{2n_x}_{++},\\
& P_{12}\in GL_{n_x},\;
N(K,P,\gamma)\preceq 0
\Bigr\}.
\end{aligned}
\end{equation}
The nonstrict inequality in \eqref{eq:Snd-def} does not by itself imply internal stability. If $(K,P,\gamma)\in\cS_{\mathrm{nd}}$ and $K\in\cC$, then the nonstrict bounded-real lemma yields $J(K)\le\gamma$ \citep{tang2023global}. We therefore distinguish the algebraic certificate set $\cS_{\mathrm{nd}}$ from its stabilizing part
\[
\cS_{\mathrm{nd}}^{\mathrm{st}}
:=\{(K,P,\gamma)\in\cS_{\mathrm{nd}}:K\in\cC\}.
\]
The associated exact nondegenerate controller set is
\begin{equation}\label{eq:Cnd-def}
\cC_{\mathrm{nd}}:=\{K\in\cC:\exists P\succ 0\text{ such that }(K,P,J(K))\in\cS_{\mathrm{nd}}\}.
\end{equation}
\citet{tang2023global} prove that every \emph{exact} Clarke-stationary controller $x \in \cC_{\mathrm{nd}}$ ($0 \in \partial f(x)$) is globally optimal. Their analysis rests on an extended convex lifting. Introduce the lifted variable
\[
z=(X,Y,M,H,F,L,\gamma),
\]
where $X,Y\in\Sbb^{n_x}$, $M\in\R^{n_x\times n_x}$, $H\in\R^{n_x\times n_y}$, $F\in\R^{n_u\times n_x}$, $L\in\R^{n_u\times n_y}$, and $\gamma\in\R$. Define the affine matrix map $\cM(z)$ as in \eqref{eq:M-def}, and the convex lifted feasible set
\begin{equation}\label{eq:F-def}
\cF:=\left\{ z:
\begin{bmatrix}X&I\\I&Y\end{bmatrix}\succ 0,
\; \cM(z)\preceq 0
\right\}.
\end{equation}
\citet{tang2023global} construct a real-analytic diffeomorphism
\begin{equation}\label{eq:phi}
\Phi:\cS_{\mathrm{nd}}\to GL_{n_x}\times \cF,
\end{equation}
whose first coordinate is the certificate block $P_{12}$, and write $\Psi=\Phi^{-1}$ for its inverse \citep[Proposition~2]{tang2023global}. We will use only the following consequence.

\begin{proposition}[ECL diffeomorphism]\label{prop:TZ-ECL}
There exists a real-analytic diffeomorphism $\Phi:\cS_{\mathrm{nd}}\to GL_{n_x}\times \cF$ with inverse $\Psi$, and the first coordinate of $\Phi(K,P,\gamma)$ equals $P_{12}$. Equivalently, if
\[
\Psi(\Xi,z)=(K,P,\gamma),
\]
then the $(1,2)$ block of $P$ equals $\Xi$. In particular, $\Psi(\Xi,z)\in\cS_{\mathrm{nd}}$ for all $(\Xi,z)\in GL_{n_x}\times\cF$. The congruence relation used to construct $\Phi$ and $\Psi$ preserves strict definiteness; thus, if $\Psi(\Xi,z)=(K,P,\gamma)$, then $\cM(z)\prec0$ if and only if $N(K,P,\gamma)\prec0$.
\end{proposition}

The result of \cite{tang2023global} on the globality of exact Clarke stationary points is qualitative rather than quantitative, since exact Clarke stationary points lack algorithmic tractability, and prior to our work, no connection between the subdifferentials and the suboptimality gap had been established. In this work, our goal is to design an algorithmic framework to find optimal solutions for the dynamic output-feedback $H_\infty$ policy search problem.

\section{Exact reconditioning and the identity gauge slice}\label{sec:reconditioning}
The first step is to show that the nondegenerate exact-certificate geometry can be normalized exactly by controller-state similarity.

\begin{lemma}[Similarity invariance]\label{lem:similarity}
For $S\in GL_{n_x}$ define
\begin{equation}\label{eq:T_S}
\mathcal T_S(K):=
\begin{bmatrix}
D_K & C_K S^{-1}\\
S B_K & S A_K S^{-1}
\end{bmatrix}.
\end{equation}
Then $\mathcal T_S(K)$ is realization-equivalent to $K$. Consequently,
\[
\mathcal T_S(K)\in\cC \iff K\in\cC,
\qquad
J\bigl(\mathcal T_S(K)\bigr)=J(K).
\]
\end{lemma}

\begin{proof}
Let $\hat \xi = S\xi$. Under this change of controller coordinates, \eqref{eq:controller1}--\eqref{eq:controller2} become
\[
\dot{\hat \xi}=SA_KS^{-1}\hat \xi + SB_Ky,
\qquad
u = C_KS^{-1}\hat \xi + D_K y,
\]
which is precisely the controller $\mathcal T_S(K)$. Therefore $K$ and $\mathcal T_S(K)$ realize the same controller transfer matrix,
\[
D_K + C_K(sI-A_K)^{-1}B_K.
\]
Since the plant is unchanged, the closed-loop transfer matrix $T_{zw}(\cdot;K)$ is identical to $T_{zw}(\cdot;\mathcal T_S(K))$. Internal stability is preserved by state similarity, and so is the $H_\infty$ norm of the closed-loop transfer matrix. Hence $K\in\cC$ if and only if $\mathcal T_S(K)\in\cC$, and $J\bigl(\mathcal T_S(K)\bigr)=J(K)$.
\end{proof}

\begin{lemma}[Exact gauge fixing]\label{lem:gauge-fix}
Let $(K,P,\gamma)\in\cS_{\mathrm{nd}}$, with
\[
P=\begin{bmatrix}P_{11}&P_{12}\\P_{12}^\transpose&P_{22}\end{bmatrix},
\qquad P_{12}\in GL_{n_x}.
\]
Set $S:=P_{12}$ and define
\[
\widehat K:=\mathcal T_S(K),
\qquad
\widehat P:=T^{-\transpose}PT^{-1},
\qquad
T:=\operatorname{diag}(I_{n_x},S).
\]
Then $(\widehat K,\widehat P,\gamma)\in\cS_{\mathrm{nd}}$ and the off-diagonal block of $\widehat P$ equals the identity, i.e., $
\widehat P_{12}=I_{n_x}$. Moreover, if $\gamma=J(K)$ then $\gamma=J(\widehat K)$ as well.
\end{lemma}

\begin{proof}
By construction, $T$ is invertible. Because $P\succ 0$, we have $\widehat P=T^{-\transpose}PT^{-1}\succ 0$. A direct block computation gives
\[
\widehat P
=
\begin{bmatrix}
P_{11} & P_{12}S^{-1}\\
S^{-\transpose}P_{12}^\transpose & S^{-\transpose}P_{22}S^{-1}
\end{bmatrix}
=
\begin{bmatrix}
P_{11} & I\\
I & S^{-\transpose}P_{22}S^{-1}
\end{bmatrix},
\]
so indeed $\widehat P_{12}=I$. Next, since $\widehat K=\mathcal T_S(K)$, Lemma~\ref{lem:similarity} implies that the closed-loop matrices transform by controller-state similarity: $A_{\mathrm{cl}}(\widehat K)=T A_{\mathrm{cl}}(K)T^{-1},\; B_{\mathrm{cl}}(\widehat K)=T B_{\mathrm{cl}}(K),\; C_{\mathrm{cl}}(\widehat K)= C_{\mathrm{cl}}(K)T^{-1}$, and $D_{\mathrm{cl}}(\widehat K)=D_{\mathrm{cl}}(K)$. Therefore
\[
N(\widehat K,\widehat P,\gamma)
=
\operatorname{diag}(T^{-\transpose},I,I)
\,N(K,P,\gamma)\,
\operatorname{diag}(T^{-1},I,I).
\]
Since $N(K,P,\gamma)\preceq 0$ and congruence by an invertible matrix preserves semidefiniteness, we obtain $N(\widehat K,\widehat P,\gamma)\preceq 0$. Thus $(\widehat K,\widehat P,\gamma)\in\cS_{\mathrm{nd}}$. Finally, Lemma~\ref{lem:similarity} yields $J(\widehat K)=J(K)$. Hence if $\gamma=J(K)$, then also $\gamma=J(\widehat K)$.
\end{proof}

Lemma~\ref{lem:gauge-fix} shows that every exact nondegenerate controller admits a realization-equivalent representative whose certificate satisfies $P_{12}=I$. This motivates the following slice as a representative of the set $\mathcal{C}_{nd}$.

\begin{definition}[Identity gauge slice]
Define
\[
\cS_I:=\{(K,P,\gamma)\in\cS_{\mathrm{nd}}: P_{12}=I_{n_x}\}.
\]
Let $
\Psi_I:\cF\to\cS_I$ defined as $\Psi_I(z):=\Psi(I_{n_x},z)$, and write $
\Psi_I(z)=\bigl(K_I(z),P_I(z),\gamma(z)\bigr)$,
where $\gamma(z)$ denotes the last coordinate of $z$.
\end{definition}

\begin{proposition}[Slice chart]\label{prop:slice-chart}
The map $\Psi_I$ is a real-analytic diffeomorphism from $\cF$ onto $\cS_I$. Define
\[
\cF_{\mathrm{st}}:=\{z\in\cF:K_I(z)\in\cC\},
\qquad
\cS_I^{\mathrm{st}}:=\cS_I\cap\cS_{\mathrm{nd}}^{\mathrm{st}}.
\]
Then $\cF_{\mathrm{st}}$ is relatively open in $\cF$, and $\Psi_I$ restricts to a real-analytic bijection, with real-analytic inverse, between $\cF_{\mathrm{st}}$ and $\cS_I^{\mathrm{st}}$. Moreover, for every exact nondegenerate controller $K\in\cC_{\mathrm{nd}}$ there exists $z\in\cF_{\mathrm{st}}$ such that $K_I(z)$ is realization-equivalent to $K$ and
\[
\Psi_I(z)=\bigl(K_I(z),P_I(z),J(K)\bigr)\in\cS_I^{\mathrm{st}}.
\]
\end{proposition}

\begin{proof}
Because $\Psi$ is a real-analytic diffeomorphism on $GL_{n_x}\times\cF$, its restriction to the embedded submanifold $\{I_{n_x}\}\times \cF$ is a real-analytic diffeomorphism onto its image. By Proposition~\ref{prop:TZ-ECL}, the first coordinate of $\Phi(K,P,\gamma)$ equals $P_{12}$. Hence $\Psi(I_{n_x},z)$ has $P_{12}=I_{n_x}$ for every $z\in\cF$, so the image is contained in $\cS_I$.

Conversely, if $(K,P,\gamma)\in\cS_I$, then $\Phi(K,P,\gamma)=(I_{n_x},z)$ for a unique $z\in\cF$, again because the first coordinate of $\Phi$ is $P_{12}$. Therefore $(K,P,\gamma)=\Psi(I_{n_x},z)=\Psi_I(z)$, proving surjectivity onto $\cS_I$. Since $\Psi_I$ is the restriction of a diffeomorphism, it is itself a real-analytic diffeomorphism. Since the Hurwitz set $\cC$ is open and $K_I$ is continuous, $\cF_{\mathrm{st}}=K_I^{-1}(\cC)$ is relatively open in $\cF$. The stated restricted bijection and its analytic inverse follow directly from the definitions.

Now let $K\in\cC_{\mathrm{nd}}$. By definition there exists $P\succ 0$ with $(K,P,J(K))\in\cS_{\mathrm{nd}}$. Applying Lemma~\ref{lem:gauge-fix}, we obtain a realization-equivalent controller $\widehat K$ and a matrix $\widehat P$ such that $(\widehat K,\widehat P,J(K))\in\cS_I$. Since $\Psi_I$ parameterizes $\cS_I$, there exists $z\in\cF$ with $
\Psi_I(z)=\bigl(\widehat K,\widehat P,J(K)\bigr)$.
Because $\widehat K$ is realization-equivalent to $K$, Lemma~\ref{lem:similarity} gives $\widehat K\in\cC$ and $J(\widehat K)=J(K)$. Hence $z\in\cF_{\mathrm{st}}$ and the displayed certificate belongs to $\cS_I^{\mathrm{st}}$.
\end{proof}

The identity gauge slice contains both exact and nonexact certificates. We isolate the exact part.

\begin{definition}[Exact slice set]
Define
\[
\cE_I:=\{z\in\cF_{\mathrm{st}}: \gamma(z)=J(K_I(z))\}.
\]
For a subset $\cZ\subset \cF_{\mathrm{st}}$, write $\cE_I(\cZ):=\cE_I\cap \cZ$.
\end{definition}

\begin{lemma}\label{lem:closed-exact}
The set $\cE_I$ is relatively closed in $\cF_{\mathrm{st}}$. Consequently, if $\cZ\subset\cF_{\mathrm{st}}$ is compact, then $\cE_I(\cZ)$ is compact.
\end{lemma}

\begin{proof}
By Proposition~\ref{prop:slice-chart}, $K_I$ is continuous on $\cF_{\mathrm{st}}$. Since $J$ is locally Lipschitz on $\cC$, it is continuous on $\cC$. The map $z\mapsto \gamma(z)$ is continuous because it is just the last coordinate projection. Therefore
\[
h(z):=\gamma(z)-J(K_I(z))
\]
is continuous on $\cF_{\mathrm{st}}$, and $\cE_I=h^{-1}(\{0\})$ is relatively closed there. If $\cZ\subset\cF_{\mathrm{st}}$ is compact, then $\cE_I(\cZ)$ is closed in $\cZ$ and hence compact.
\end{proof}

\section{Quantitative first-order error bounds on the exact slice}
\label{sec:quantitative}

Now, we present our main result, which connects the suboptimality of an exact slice member to its set of Clarke subdifferentials, thereby providing a backbone for the analysis of nonsmooth optimization algorithms.

\begin{theorem}[Suboptimality Gap]\label{thm:pairwise}
Let $z,z^\star\in\cE_I$, and
\[
\Psi_I(z)=\bigl(K,P,J(K)\bigr),
\qquad
\Psi_I(z^\star)=\bigl(K^\star,P^\star,J(K^\star)\bigr).
\]
Then
\begin{equation}\label{eq:pairwise-exact}
J(K)-J(K^\star)
\le
\distop\!\bigl(0,\partial J(K)\bigr)
\cdot
\bigl\|D K_I(z)[z^\star-z]\bigr\|.
\end{equation}
In particular,
\begin{equation*}
J(K)-J(K^\star)
\le
\opnorm{D K_I(z)}\,\|z^\star-z\|\,\distop\!\bigl(0,\partial J(K)\bigr).
\end{equation*}
\end{theorem}

\vspace{-0.1cm}
\begin{proof}
Fix $z,z^\star\in\cE_I$ and define the segment
\[
z_t:=z+t(z^\star-z),\qquad t\in[0,1].
\]
Since $\cF$ is convex, $z_t\in\cF$ for every $t\in[0,1]$. Let $
\Psi_I(z_t)=\bigl(K_t,P_t,\gamma_t\bigr). $
Since $z\in\cE_I\subset\cF_{\mathrm{st}}$ and $\cF_{\mathrm{st}}$ is relatively open in $\cF$, there exists $t_0\in(0,1]$ such that $z_t\in\cF_{\mathrm{st}}$ for every $t\in[0,t_0]$. For such $t$, $K_t\in\cC$ and $\Psi_I(z_t)\in\cS_I\subset\cS_{\mathrm{nd}}$, so the nonstrict bounded-real lemma yields
\begin{equation}\label{eq:bounded-real-path}
J(K_t)\le \gamma_t,\qquad \forall t\in[0,t_0].
\end{equation}
Moreover, the $\gamma$-coordinate is affine along the segment:
\[
\gamma_t=(1-t)\gamma(z)+t\gamma(z^\star).
\]
Since $z,z^\star\in\cE_I$, exactness gives $
\gamma(z)=J(K), \gamma(z^\star)=J(K^\star).$ Combining this with \eqref{eq:bounded-real-path}, we obtain
\begin{align}
\limsup_{t\downarrow 0}\frac{J(K_t)-J(K)}{t}
&\le \lim_{t\downarrow 0}\frac{\gamma_t-\gamma(z)}{t} \notag\\
&= \gamma(z^\star)-\gamma(z) \notag\\
&= J(K^\star)-J(K). \label{eq:curve-upper}
\end{align}

Next, since $K_I$ is real analytic on $\cF$, it is differentiable at $z$ and $
K_t=K + tV + r(t)$,
with
\[
V:=D K_I(z)[z^\star-z], \quad \frac{\|r(t)\|}{t}\to 0\quad(t\downarrow 0).
\]
Because $J$ is locally Lipschitz at $K$, there exist $L>0$ and a neighborhood of $K$ on which
\[
|J(K_1)-J(K_2)|\le L\|K_1-K_2\|.
\]
For all sufficiently small $t>0$ we therefore have
\begin{align*}
\hspace{-0.35cm}
\frac{J(K+tV)-J(K)}{t}
&\le \frac{J(K_t)-J(K)}{t} + \frac{|J(K+tV)-J(K_t)|}{t}\\
&\le \frac{J(K_t)-J(K)}{t} + L\frac{\|K+tV-K_t\|}{t}\\
&= \frac{J(K_t)-J(K)}{t} + L\frac{\|r(t)\|}{t}.
\end{align*}
Taking the limit superior as $t\downarrow 0$ and using \eqref{eq:curve-upper} yields
\begin{equation}\label{eq:dir-upper}
\hspace{-0.3cm}
J'(K;V)
=\lim_{t\downarrow 0}\frac{J(K+tV)-J(K)}{t}
\le J(K^\star)-J(K),
\end{equation}
where the existence of the directional derivative follows from Proposition~\ref{prop:TZ-regular}. Again, by Proposition~\ref{prop:TZ-regular},
\[
J'(K;V)=\max_{G\in\partial J(K)}\inner{G}{V}.
\]
The Clarke subdifferential of a locally Lipschitz function is nonempty, compact, and convex, so there exists $G_{\min}\in\partial J(K)$ with $
\|G_{\min}\|=\distop\!\bigl(0,\partial J(K)\bigr)$. Hence
\[
J'(K;V) 
& \ge \inner{G_{\min}}{V} \\
& \ge -\|G_{\min}\|\,\|V\| \\
& = -\distop\!\bigl(0,\partial J(K)\bigr)\,\|V\|. \numberthis{eq:dir-lower}
\]
Combining \eqref{eq:dir-upper} and \eqref{eq:dir-lower},
\[
J(K)-J(K^\star)
\le \distop\!\bigl(0,\partial J(K)\bigr)\,\|V\|,
\]
which is \eqref{eq:pairwise-exact}. Finally,
\[
\|V\|=\bigl\|D K_I(z)[z^\star-z]\bigr\| \le \opnorm{D K_I(z)}\,\|z^\star-z\|,
\]
which completes the proof.
\end{proof}

Theorem~\ref{thm:pairwise} is the quantitative refinement of the landscape argument in \citet{tang2023global} on the reconditioned exact slice. It becomes especially useful on compact convex slices.

\begin{corollary}\label{cor:compact-slice}
Let $\cZ\subset\cF_{\mathrm{st}}$ be compact and convex, and assume that $\cE_I(\cZ)\neq\varnothing$. Define
\[
J^{\mathrm{ex}}_{\cZ}:=\min_{z\in\cE_I(\cZ)} J\bigl(K_I(z)\bigr),
\]
which is well defined by Lemma~\ref{lem:closed-exact}. Also define
\[
L_{\cZ}:=\sup_{z\in\cZ}\opnorm{D K_I(z)}, \;
D_{\cZ}:=\diamop(\cZ), \;
C_{\cZ}:=L_{\cZ}D_{\cZ}.
\]
Then for every $z\in\cE_I(\cZ)$,
\begin{equation}\label{eq:compact-slice-bound}
J\bigl(K_I(z)\bigr)-J^{\mathrm{ex}}_{\cZ}
\le
C_{\cZ}\,\distop\!\bigl(0,\partial J(K_I(z))\bigr).
\end{equation}
\end{corollary}

\begin{proof}
By Lemma~\ref{lem:closed-exact}, the set $\cE_I(\cZ)$ is compact. Since $z\mapsto J(K_I(z))$ is continuous on $\cE_I(\cZ)$, there exists $z^\star\in\cE_I(\cZ)$ such that
\[
J\bigl(K_I(z^\star)\bigr)=J^{\mathrm{ex}}_{\cZ}.
\]
Applying Theorem~\ref{thm:pairwise} to $z$ and $z^\star$ gives
\[
J\bigl(K_I(z)\bigr)-J^{\mathrm{ex}}_{\cZ}
\le
\opnorm{D K_I(z)}\,\|z-z^\star\|\,\distop\!\bigl(0,\partial J(K_I(z))\bigr).
\]
Since $z,z^\star\in\cZ$, we have $\|z-z^\star\|\le D_{\cZ}$ and $\opnorm{D K_I(z)}\le L_{\cZ}$. Substituting these bounds yields \eqref{eq:compact-slice-bound} and proof is concluded.
\end{proof}

This result can serve as a core for the analysis of nonsmooth optimization algorithms for output-feedback $H_\infty$ policy search. If $\cE_I(\cZ)$ contains a globally optimal controller for \eqref{eq:policyopt}, then $J^{\mathrm{ex}}_{\cZ}=J^*=\inf_{K\in\cC}J(K)$, and every Clarke-stationary $z\in\cE_I(\cZ)$ is globally optimal. In addition, as we will show in the next section, the convex lift $\gamma(z)$ attains values arbitrarily close to $J^*$. As a result, we can establish a similar connection between global suboptimality and the Clarke subdifferential. We do not include this result in the interest of space.

In turn, we discuss algorithmic consequence of the value gap bounds.

\begin{corollary}\label{cor:algorithmic}
Let $\cZ\subset\cF_{\mathrm{st}}$ be compact and convex with $\cE_I(\cZ)\neq\varnothing$, and let $\{z_k\}_{k\ge 0}\subset \cE_I(\cZ)$ be any sequence. Define the stationarity residual
\[
\varepsilon_k:=\distop\!\bigl(0,\partial J(K_I(z_k))\bigr).
\]
Then
\begin{equation}\label{eq:alg-gap}
J\bigl(K_I(z_k)\bigr)-J^{\mathrm{ex}}_{\cZ}
\le C_{\cZ}\,\varepsilon_k,
\qquad \forall k\ge 0.
\end{equation}
In particular, if $\varepsilon_k\to 0$, then
\[
J\bigl(K_I(z_k)\bigr)\to J^{\mathrm{ex}}_{\cZ}.
\]
If $\cE_I(\cZ)$ contains a globally optimal controller, then \eqref{eq:alg-gap} is a global optimality-gap estimate.
\end{corollary}

\begin{proof}
This is an immediate specialization of Corollary~\ref{cor:compact-slice} to the points $z_k$.
\end{proof}

Corollary~\ref{cor:algorithmic} is deliberately abstract. It does not prove that a particular Goldstein, gradient-sampling, or bundle method generates exact iterates in a compact slice, nor does it establish a complexity bound for driving $\varepsilon_k$ to zero. These are separate algorithmic questions. What the corollary does provide is the missing bridge from approximate stationarity to value suboptimality \emph{once} a reconditioned method is analyzed on the exact slice. This is the dynamic output-feedback counterpart (and hence more general) of the value-gap open problem explicitly raised in the state-feedback case by \citet{guo2022global}, now considered on the nondegenerate ECL manifold introduced by \citet{tang2023global}.

It is widely known that even $\varepsilon$-Clarke stationarity faces computational intractability~\citep{zhang2020complexity,guo2022global}. Therefore, in order to obtain concrete convergence rates, we turn to more relaxed solution concepts such as $(\delta, \varepsilon)$-Goldstein stationarity. While in general $\distop\!\bigl(0,\partial_\delta J(K) \bigr) \leq \distop\!\bigl(0,\partial J(K) \bigr)$, under an outer continuity assumption on the Clarke subdifferential map $K \mapsto \partial J(J)$, we can recover a reverse inequality up to an additive term. This, in turn, allows us to connect off-the-shelf convergence guarantees of nonsmooth optimization algorithms such as \citet{zhang2020complexity,davis2022gradient} for finding Goldstein stationary points to convergence of the objective value for the $H_\infty$ policy search problem. We formalize this result in the following.

\begin{assumption} \label{def:Hausdorff}
    The Clarke subdifferential map $K \mapsto \partial J(K)$ is one-sided Hausdorff continuous; that is, for each $K \in \mathcal{C}_{nd}$, there exists a nondecreasing modulus $\psi_K : [0, \infty) \mapsto [0, \infty)$, with $\omega_K(r) \rightarrow 0$ as $r \downarrow 0$, such that
    \[
    \partial J(K') \subset \partial J(K) + \psi_K(\| K' - K \|) \mathbb{B}_2(0, 1),
    \]
    for all $K'$ near $K$.
\end{assumption}

A commonly used special case of Assumption~\ref{def:Hausdorff} is the linear one, where $\psi_K(\| K' - K \|) = a \| K' - K \|$ for a positive constant $a$, which \citet{kong2025lipschitz} refer to as upper Lipschitzness of the subdifferential.

In turn, we relate Assumption~\ref{def:Hausdorff} to a relationship between $\distop\!\bigl(0,\partial_\delta J(K) \bigr)$ and $\distop\!\bigl(0,\partial J(K) \bigr)$.

\begin{proposition}\label{prop:clarke_to_gold}
    Under Assumption~\ref{def:Hausdorff}, we have
    \[
    \distop\!\bigl(0,\partial J(K) \bigr) \leq \distop\!\bigl(0,\partial_\delta J(K) \bigr) + \psi_K(\delta).
    \]
\end{proposition}
\begin{proof}
    Because the set $\partial J(K) + \psi_K(\| K' - K \|) \mathbb{B}_2(0, 1)$ is convex and contains every $\partial J(K')$ for all $K' \in \mathbb{B}_2(K, \delta)$, we have
    \[
    \partial J_\delta(K) \subset \partial J(K) + \psi_K(\delta) \mathbb{B}_2(0, 1).
    \]
    Therefore,
    \[
    \distop\!\bigl(0,\partial J(K) \bigr) - \psi_K(\delta) \leq \distop\!\bigl(0,\partial_\delta J(K) \bigr),
    \]
    and the proof is complete.
\end{proof}

\begin{theorem}[Convergence rate]
    Under Assumption~\ref{def:Hausdorff}, the Interpolated Normalized Gradient Descent (INGD) algorithm~\citep{zhang2020complexity} takes $\tilde{\mathcal{O}}(\dfrac{1}{ \delta \varepsilon^3})$ gradient and function valuations to find a $z \in \cE_I$ such that $K_I(z)$ is an $(\varepsilon, \delta)$-Goldstein stationary point and more importantly $K_I(z)$ is $C_{\cZ}\, \mleft(\varepsilon + \psi_K(\delta)\mright)$-suboptimal on the set $\mathcal{Z}$,
    \[
    J\bigl(K_I(z)\bigr)-J^{\mathrm{ex}}_{\cZ}
    \le
    C_{\cZ}\, \mleft(\varepsilon + \psi_K(\delta)\mright).
    \]
\end{theorem}
\begin{proof}
    The proof follows immediately by combining Proposition~\ref{prop:clarke_to_gold}, Corollary~\ref{cor:algorithmic}, and the analysis of INGD from \citet[Theorem~3.2]{davis2022gradient}.
\end{proof}

\section{Global value computation and strict controller recovery in the convex lift}\label{sec:global-value}
The exact-slice argument of Section~\ref{sec:quantitative} is intrinsically local as it converts first-order information at an exact point into a value gap relative to an exact comparison point. The value equivalence used in this section is an established consequence of extended convex lifting (ECL). In particular, \citet[Theorem~2.1]{zheng2025extended} show that a policy objective equipped with an ECL has the same infimum as its convex lifted problem, and \citet[Corollary~4.1]{zheng2025extended} state this conclusion explicitly for dynamic output-feedback $H_\infty$ control. We use this identity to distinguish nonstrict value computation from strict stabilizing-controller recovery and to formulate the resulting feasibility-bisection procedure. Although a nonstrict point of $\cF$ need not decode to an internally stabilizing controller, strict lifted points do support stabilizing controller recovery.

\subsection{Strict lifted points}

Define $
\cF_{\mathrm{str}}:=\{z\in\cF:\cM(z)\prec0\}$.

\begin{lemma}\label{lem:strict-lifted}
For every $z\in\cF_{\mathrm{str}}$,
\[
\Psi_I(z)=\bigl(K_I(z),P_I(z),\gamma(z)\bigr)\in \cS_I\subset \cS_{\mathrm{nd}},
\]
and $K_I(z)\in\cC$ with
\[
J\bigl(K_I(z)\bigr)<\gamma(z).
\]
\end{lemma}

\begin{proof}
By Proposition~\ref{prop:slice-chart}, $\Psi_I(z)\in\cS_I\subset\cS_{\mathrm{nd}}$. The strict-definiteness correspondence in Proposition~\ref{prop:TZ-ECL} gives
\[
N\bigl(K_I(z),P_I(z),\gamma(z)\bigr)\prec0.
\]
Its leading principal block therefore satisfies
\[
A_{\mathrm{cl}}(K_I(z))^\top P_I(z)+P_I(z)A_{\mathrm{cl}}(K_I(z))\prec0.
\]
Since $P_I(z)\succ0$, the Lyapunov theorem implies $K_I(z)\in\cC$. The strict bounded-real lemma then gives $J(K_I(z))<\gamma(z)$.
\end{proof}

\subsection{Strict nondegenerate approximation of stabilizing controllers}

\begin{lemma}\label{lem:strict-brl}
Let $K\in\cC$ and let $\bar\gamma>J(K)$. Then there exists $P\in\Sbb^{2n_x}_{++}$ such that
\[
N(K,P,\bar\gamma)\prec 0.
\]
\end{lemma}

\begin{proof}
This is the standard strict bounded-real lemma in continuous time; see also the lemma invoked in the proof of Proposition~4 in \citet{tang2023global}.
\end{proof}

\begin{lemma}[Strict nondegenerate approximation]\label{lem:strict-approx}
Let $K\in\cC$ and let $\eta>0$. Then there exists $z\in\cF_{\mathrm{str}}$ such that $K_I(z)$ is realization-equivalent to $K$ and
\[
\gamma(z)\le J(K)+\eta.
\]
In fact, one may arrange $\gamma(z)=J(K)+\eta$.
\end{lemma}

\begin{proof}
Set $\bar\gamma:=J(K)+\eta$. By Lemma~\ref{lem:strict-brl}, there exists $P\succ0$ such that
\[
N(K,P,\bar\gamma)\prec0.
\]
Because positive definiteness and strict negative definiteness are open conditions, there exists $\delta>0$ such that whenever $E\in\R^{n_x\times n_x}$ satisfies $\|E\|<\delta$, the symmetric perturbation
\[
\widetilde P:=P+\begin{bmatrix}0&E\\E^\transpose&0\end{bmatrix}
\]
still obeys
\[
\widetilde P\succ0,
\qquad
N(K,\widetilde P,\bar\gamma)\prec0.
\]
Since $GL_{n_x}$ is open and dense in $\R^{n_x\times n_x}$, we may choose such an $E$ with $\widetilde P_{12}=P_{12}+E\in GL_{n_x}$. Hence
\[
(K,\widetilde P,\bar\gamma)\in\cS_{\mathrm{nd}}.
\]

Applying Lemma~\ref{lem:gauge-fix}, we obtain a controller $\widehat K$ realization-equivalent to $K$ and a matrix $\widehat P$ such that
\[
(\widehat K,\widehat P,\bar\gamma)\in\cS_I,
\qquad
N(\widehat K,\widehat P,\bar\gamma)\prec0,
\]
where strictness is preserved by the congruence in the proof of Lemma~\ref{lem:gauge-fix}. Now Proposition~\ref{prop:slice-chart} yields a unique $z\in\cF$ satisfying
\[
\Psi_I(z)=\bigl(\widehat K,\widehat P,\bar\gamma\bigr).
\]
By Proposition~\ref{prop:TZ-ECL}, $\cM(z)\prec0$, so $z\in\cF_{\mathrm{str}}$. Thus $K_I(z)=\widehat K$ is realization-equivalent to $K$, and
\[
\gamma(z)=\bar\gamma=J(K)+\eta,
\]
which proves the claim.
\end{proof}

\subsection{Value identity}

Now, we reiterate the value identity of ECL which indicates that $
\inf_{z\in\cF}\gamma(z) = 
\inf_{K\in\cC}J(K)$.

\begin{proposition}[Lifted-value identity]\label{prop:global-value}
Assume $\cC\neq\varnothing$. Then the global optimal value of \eqref{eq:policyopt} satisfies
\begin{equation}
J^\star:=\inf_{K\in\cC}J(K)=\inf_{z\in\cF}\gamma(z)=:J_{\mathrm{lift}}.
\label{eq:global-value}
\end{equation}
No attainment assumption on a prescribed exact slice is required.
\end{proposition}

\begin{proof}
This is an application of the general extended-convex-lifting value theorem of \citet[Theorem~2.1]{zheng2025extended}; its dynamic output-feedback $H_\infty$ specialization is stated in \citet[Corollary~4.1]{zheng2025extended}. For the generalized-plant coordinates used here, Proposition~\ref{prop:TZ-ECL} supplies the lifting diffeomorphism and preserves the $\gamma$ coordinate. Lemma~\ref{lem:strict-approx} supplies strict lifted certificates arbitrarily close to every stabilizing-controller cost, while convex interpolation between any point of $\cF$ and a strict point of $\cF_{\mathrm{str}}$ supplies the required closure inclusion for the nonstrict lift. Thus the hypotheses of the cited value theorem hold and yield \eqref{eq:global-value}.
\end{proof}

\begin{corollary}\label{cor:lift-attain}
Suppose there exists $z^\star\in\cF_{\mathrm{st}}$ such that $
\gamma(z^\star)=\min_{z\in\cF}\gamma(z)=J_{\mathrm{lift}}$. Then $K_I(z^\star)$ is globally optimal and $
J\bigl(K_I(z^\star)\bigr)=\gamma(z^\star)=J^\star$.
\end{corollary}

\begin{proof}
Since $z^\star\in\cF_{\mathrm{st}}$, $K_I(z^\star)\in\cC$. Also $\Psi_I(z^\star)\in\cS_I\subset\cS_{\mathrm{nd}}$, so the nonstrict bounded-real lemma gives
\[
J\bigl(K_I(z^\star)\bigr)\le \gamma(z^\star)=J_{\mathrm{lift}}=J^\star.
\]
Since $J^\star$ is the infimum of $J$ over $\cC$, the reverse inequality is automatic. Hence equality holds throughout.
\end{proof}

\vspace{-0.1cm}
\begin{remark}\label{rem:attain}
Proposition~\ref{prop:global-value} records a statement about the optimal value, not about pointwise stable decoding of the full nonstrict lift. A minimizer in $\cF$ need not lie in $\cF_{\mathrm{st}}$. If such a boundary minimizer $z^\star$ exists, then convex combinations with any $z^\circ\in\cF_{\mathrm{str}}$ yield strict lifted points whose decoded stabilizing-controller costs converge to $J^\star$. Degeneracy may also prevent the infimum from being attained by a point of $\cF$, but neither phenomenon affects the value identity $J^\star=J_{\mathrm{lift}}$. Nondegeneracy and exactness re-enter when one wants local first-order certificates on the identity slice.
\end{remark}

\subsection{Bisection over the convex lift}

For a fixed scalar $\bar\gamma\in\R$, consider the feasibility problems
\[
\mathrm{Feas}(\bar\gamma):\quad
\exists (X,Y,M,H,F,L)\ \text{such that}\ 
\]
\[
\begin{bmatrix}X&I\\I&Y\end{bmatrix}\succ0, \quad \textrm{and} \quad
\cM(X,Y,M,H,F,L,\bar\gamma)\preceq0, 
\]
and
\[
\mathrm{StrictFeas}(\bar\gamma):\quad
\exists (X,Y,M,H,F,L)\ \text{such that}\ 
\]
\[
\begin{bmatrix}X&I\\I&Y\end{bmatrix}\succ0, \quad \textrm{and} \quad
\cM(X,Y,M,H,F,L,\bar\gamma)\prec0.
\]
These are convex feasibility problems because the constraints are affine in the decision variables $(X,Y,M,H,F,L)$.

\begin{proposition}[Monotone feasibility threshold]\label{prop:feas-threshold}
The feasibility predicate $\mathrm{Feas}(\bar\gamma)$ is monotone nondecreasing in $\bar\gamma$: if $\mathrm{Feas}(\bar\gamma)$ holds and $\bar\gamma'\ge \bar\gamma$, then $\mathrm{Feas}(\bar\gamma')$ also holds. Moreover,
\begin{align*}
    & \mathrm{Feas}(\bar\gamma)\implies \bar\gamma\ge J^\star, \\
    &  \bar\gamma>J^\star\implies \mathrm{StrictFeas}(\bar\gamma)\implies \mathrm{Feas}(\bar\gamma).
\end{align*}

\end{proposition}

\begin{proof}
If $\mathrm{Feas}(\bar\gamma)$ holds, let $(X,Y,M,H,F,L)$ be a feasible point. Replacing $\bar\gamma$ by a larger value $\bar\gamma'\ge \bar\gamma$ only subtracts the positive semidefinite block $(\bar\gamma'-\bar\gamma)\operatorname{diag}(0,0,I,I)$ from $\cM$, so feasibility is preserved. This proves monotonicity.

If $\mathrm{Feas}(\bar\gamma)$ holds, then some $z\in\cF$ has $\gamma(z)=\bar\gamma$. Proposition~\ref{prop:global-value} therefore gives
\[
\bar\gamma\ge \inf_{\widetilde z\in\cF}\gamma(\widetilde z)=J^\star.
\]

Conversely, suppose $\bar\gamma>J^\star$. By the definition of the infimum, there exists $K\in\cC$ with $J(K)<\bar\gamma$. Applying Lemma~\ref{lem:strict-approx} with $\eta=\bar\gamma-J(K)$ produces $z\in\cF_{\mathrm{str}}$ with $\gamma(z)=\bar\gamma$. Hence $\mathrm{StrictFeas}(\bar\gamma)$ holds, and strict feasibility trivially implies nonstrict feasibility.
\end{proof}

\begin{theorem}[Convex-lift bisection algorithm]\label{thm:bisection}
Assume $\cC\neq\varnothing$ and fix $\varepsilon>0$. Consider the following procedure.
\begin{enumerate}
    \item Starting from $m=0$, test $\mathrm{Feas}(2^m)$ until the first feasible index $m=m_0$ is found.
    \item Set $L_0:=0$ and $U_0:=2^{m_0}$.
    \item Given $[L_j,U_j]$, form $G_j:=(L_j+U_j)/2$.
    \begin{itemize}[leftmargin=1em]
        \item If $\mathrm{Feas}(G_j)$ holds, set $L_{j+1}:=L_j$ and $U_{j+1}:=G_j$.
        \item If $\mathrm{Feas}(G_j)$ fails, set $L_{j+1}:=G_j$ and $U_{j+1}:=U_j$.
    \end{itemize}
    \item Stop when $U_j-L_j\le\varepsilon/2$, set
    \[
    \widehat\gamma:=U_j+\frac{\varepsilon}{2},
    \]
    solve $\mathrm{StrictFeas}(\widehat\gamma)$ to obtain $\widehat z\in\cF_{\mathrm{str}}$ with $\gamma(\widehat z)=\widehat\gamma$, and output
    \[
    K_\varepsilon:=K_I(\widehat z).
    \]
\end{enumerate}
Then the procedure uses at most
$
m_0+1+
\max\left\{
0,
\left\lceil
\log_2\!\left(\frac{2U_0}{\varepsilon}\right)
\right\rceil
\right\}
$
nonstrict feasibility tests and one strict-feasibility test. In
particular, the total number of feasibility-oracle calls is
$
\mathcal{O}\!\left(
1+\log\max\{1,J^\star\}
+\log\max\{1,\varepsilon^{-1}\}
\right).
$
On termination,
\[
J^\star\le U_j\le J^\star+\frac{\varepsilon}{2},
\qquad
K_\varepsilon\in\cC,
\qquad
J(K_\varepsilon)<\widehat\gamma\le J^\star+\varepsilon.
\]
Hence $K_\varepsilon$ is an $\varepsilon$-optimal stabilizing controller.
\end{theorem}

\begin{proof}
Since $\cC\neq\varnothing$, Proposition~\ref{prop:global-value}
gives $J^\star<\infty$. Therefore $2^m>J^\star$ for all
sufficiently large $m$, and Proposition~\ref{prop:feas-threshold}
guarantees that the doubling phase terminates after $m_0+1$
feasibility tests.

Every interval $[L_j,U_j]$ constructed by the algorithm satisfies
$L_j\le J^\star\le U_j$. The base case follows from
$L_0=0\le J^\star$ and feasibility of $U_0$. If
$\mathrm{Feas}(G_j)$ holds, Proposition~\ref{prop:feas-threshold}
gives $J^\star\le G_j=U_{j+1}$; if it fails, the implication
$\bar\gamma>J^\star\Rightarrow\mathrm{Feas}(\bar\gamma)$ gives
$G_j\le J^\star$, so $L_{j+1}=G_j$. Thus the invariant is
preserved.

After $n$ bisection steps, the interval length is $U_0/2^n$.
Hence the stopping condition is reached after at most
$
\max\left\{
0,
\left\lceil
\log_2\!\left(\frac{2U_0}{\varepsilon}\right)
\right\rceil
\right\}
$
bisection steps. The algorithm then performs one final
strict-feasibility test.

At termination,
$
J^\star\le U_j\le L_j+\frac{\varepsilon}{2}
\le J^\star+\frac{\varepsilon}{2}.
$
Consequently,
$\widehat\gamma=U_j+\frac{\varepsilon}{2}>J^\star$ and
$\widehat\gamma\le J^\star+\varepsilon$.
Proposition~\ref{prop:feas-threshold} guarantees
$\mathrm{StrictFeas}(\widehat\gamma)$. Lemma~\ref{lem:strict-lifted}
applied to the recovered strict point gives $K_\varepsilon\in\cC$
and
$
J(K_\varepsilon)<\widehat\gamma\le J^\star+\varepsilon.
$
Thus $K_\varepsilon$ is an $\varepsilon$-optimal stabilizing
controller.
\end{proof}

\section{Conclusion}
We gave a quantitative and algorithmic treatment of nonsmooth dynamic output-feedback $H_\infty$ policy search via the extended convex lift. On the stable exact identity-gauge slice, we proved that approximate stationarity controls suboptimality on compact exact slices. For global computation, we used the established lifted-value identity to combine nonstrict-feasibility bisection with one strict-feasibility recovery step, producing an explicit $\varepsilon$-optimal stabilizing controller. Future work includes weakening the outer-continuity assumption on the Clarke subdifferential, analyzing concrete algorithms that operate near the exact slice without requiring exact iterates, and extending the framework to structured or reduced-order controllers and data-driven robust control.

\section*{ACKNOWLEDGMENT}
AS and PJ were partially supported by the ONR grant N00014-24-1-2470. Part of this research was performed while the AS was visiting the Institute for Mathematical and Statistical Innovation (IMSI), which is supported by the National Science Foundation (Grant No. DMS-2425650).

\bibliographystyle{IEEEtranN}
\bibliography{references}

\end{document}